\input ./style/arxiv-vmsta.cfg
\documentclass[numbers,compress,v1.0.1]{vmsta}

\volume{2}
\issue{4}
\pubyear{2015}
\firstpage{391}
\lastpage{399}
\doi{10.15559/15-VMSTA42}% Updated by VTEXPTS2LaTeX.exe, 17.12.2015
\startlocaldefs

\newcommand{\rrvert}{\vert}
\newcommand{\llvert}{\vert}
\allowdisplaybreaks

\newtheorem{thm}{Theorem}

\theoremstyle{definition}

\endlocaldefs

\begin{document}
\begin{frontmatter}

\title{Distance between exact and approximate distributions of partial
maxima under power normalization}

\author{\inits{A.S.}\fnm{Attahalli Shivanarayanaprasad}\snm{Praveena}\corref{cor1}\fnref{f1}}\email
{praveenavivek@gmail.com}
\cortext[cor1]{Corresponding author.}
\fntext[f1]{Research work supported by Post Doctoral Fellowship for
Women of the University Grants Commission (UGC), New Delhi. Ref:
No.F. 15-1/2011-12/PDFWM-2011-12-GE-KAR-2333(SA-II) dated 1 Nov 2013.}
\author{\inits{S.}\fnm{Sreenivasan}\snm{Ravi}\fnref{f2}}\email
{ravi@statistics.uni-mysore.ac.in}
\fntext[f2]{Research work supported by UGC Major Research Project F.No.
43-541/2014(SR) dated 16.10.2015}

%
%\author[]{\inits{}\fnm{}\snm{}\corref{cor1}}\email{}
%\cortext[cor1]{Corresponding author.}
%
%\author[]{\inits{}\fnm{}\snm{}}\email{}

\address{Department of Studies in Statistics,
University of Mysore,
Manasagangotri, Mysuru~570006, India}

%\fnref{f1}
%\fntext[]{Some remarks}

%\address[]{}
%\address[]{}

\markboth{A.S. Praveena, S. Ravi}{Distance between exact and
approximate distributions of partial maxima under power normalization}

\begin{abstract}
We obtain the distance between the exact and approximate distributions
of partial maxima of a random sample under power normalization. It is
observed that the Hellinger distance and variational distance between
the exact and approximate distributions of partial maxima under power
normalization is the same as the corresponding distances under linear
normalization.
\end{abstract}


%\begin{keyword} . \sep.
%\MSC[2010] . \sep.
%\end{keyword}

%
\received{7 October 2015}% Updated by VTEXPTS2LaTeX.exe, 17.12.2015
%15:21
%
\revised{4 December 2015}% Updated by VTEXPTS2LaTeX.exe, 17.12.2015
%15:21
%
\accepted{4 December 2015}% Updated by VTEXPTS2LaTeX.exe, 17.12.2015
%15:21
\publishedonline{23 December 2015}
\end{frontmatter}

\section{Introduction}
Let $X_1,X_2,\dots, X_n $ be independent and identically distributed (iid)
random variables with common distribution function (df) $\;F\;$ and
$M_n=\max(X_1,X_2,\dots, X_n)$, $n\geq1$.
Then $F$ is said to belong to the max domain of attraction of a
nondegenerate df $H$ under power normalization
(denoted by $F\in\mathcal{D}_p (H)$) if, for $n\ge1$, there exist constants $\alpha
_n > 0$, $\beta_n > 0$, such that
\begin{equation}
\label{Intro1} \lim_{n\rightarrow\infty}\; P \biggl(\biggl\llvert
\frac{M_n}{\alpha_n}\biggr\rrvert ^{\frac{1}{\beta_n}}\; \mbox{sign}(M_n) \le
x \biggr)= H(x), \quad x\in\mathcal C(H),
\end{equation}
the set of continuity points of $H$, where sign$(x)=-1, 0$, or $1$
according as $x<0$, $=0$, or $>0$. The limit df $H$ in (\ref{Intro1}) is
called a {\sl p}-max stable law, and we refer to \cite{mohan1993max}
for details.\medskip

\noindent{\bf The {\sl \textbf{p}}-max stable laws.} Two dfs $F$ and $G$ are said to be
of the same {\sl p}-type if $F(x) = G(A\mid x\mid^{B} \mbox
{sign}(x)), x \in R$, for some positive constants $A, B$. The {\sl
p}-max stable laws are {\sl p}-types of one of the following six laws
with parameter \(\alpha>0\):
\begin{align*}
H_{1,\alpha}(x)&= \left\{ %
\begin{array}{@{}ll}
0 & \text{if\ $x\le1$},\\
\exp\{-(\log x)^{-\alpha}\} & \text{if\ $1<x$};
\end{array} %
\right.
\\
H_{2,\alpha}(x)&= \left\{ %
\begin{array}{@{}ll}
0 & \text{if\ $x<0$},\\
\exp\{-(-\log x)^\alpha\} & \text{if\ $0\le x<1$},\\
1 & \text{if\ $1\le x$};
\end{array} %
\right.
\\
% \end{eqnarray*}
%\begin{eqnarray*}
H_3(x)&= \left\{ %
\begin{array}{@{}ll}
0 & \text{if\ $x\leq0$},\\
e^{-\frac{1}{x}} & \text{if\ $0<x$};
\end{array} %
\right.
\\
H_{4,\alpha}(x)&= \left\{ %
\begin{array}{@{}ll}
0 & \text{if\ $x\le-1$},\\
\exp\{-(-\log(-x))^{-\alpha}\} & \text{if\ $-1 < x<0$},\\
1 & \text{if\ $0\le x$};
\end{array} %
\right.
\\
H_{5,\alpha}(x) &= \left\{ %
\begin{array}{@{}ll}
\exp\{-(\log(-x))^\alpha\} & \text{if\ $x<-1$},\\
1 & \text{if\ $-1 \le x$};
\end{array} %
\right.
\\
H_6(x)&= \left\{ %
\begin{array}{@{}ll}
e^{x}& \text{if\ $x\leq0$},\\
1 & \text{if\ $0< x$}.
\end{array} %
\right.
\end{align*}
Note that $H_{2,1}(\cdot)$ is the uniform distribution over $(0,1)$.
Necessary and sufficient conditions for a df $F$ to belong to $\mathcal
{D}_p(H)$ for each
of the six {\sl p}-types of {\sl p}-max stable laws were given in \cite
{mohan1993max} (see also \cite{falk2010laws}).

As in \cite{Praveena}, we define the generalized log-Pareto
distribution (glogPd) as $W(x)=1+\log H(x)$ for $x$ with $1/e\leq
H(x)\leq1$, where $H$ is a {\sl p}-max stable law, and the distribution
functions \(W\) are given by
\begin{align*}
W_{1,\alpha}(x) &= \left\{ %
\begin{array}{@{}ll}
0 & \text{if\ $x<e$},\\
1-(\log x)^{-\alpha} & \text{if\ $e\leq x$};\\
\end{array} %
\right.
\\
W_{2,\alpha}(x) &= \left\{ %
\begin{array}{@{}ll}
0 & \text{if\ $x<e^{-1}$},\\
1-(-\log x)^{\alpha} & \text{if\ $e^{-1} \leq x<1$},\\
1 &\text{if\ $1<x$};
\end{array} %
\right.
\\
W_3(x) &= \left\{ %
\begin{array}{@{}ll}
0 & \text{if\ $ x\leq1 $},\\
1-\frac{1}{x} & \text{if\ $1<x$};
\end{array} %
\right.
\\
W_{4,\alpha}(x) &= \left\{ %
\begin{array}{@{}ll}
0 &\text{if\ $x<-e^{-1}$},\\
1-(-\log(-x))^{-\alpha} & \text{if\ $-e^{-1}\leq x<0$},\\
1 & \text{if\ $0<x$};
\end{array} %
\right.
\\
W_{5,\alpha}(x) &= \left\{ %
\begin{array}{@{}ll}
0 & \text{if\ $x<-e$},\\
1-(\log(-x))^{\alpha} & \text{if\ $-e\leq x<-1$},\\
1 & \text{if\ $-1<x$};
\end{array} %
\right.
\\
W_6(x) &=  \left\{ %
\begin{array}{@{}ll}
0 & \text{if\ $x<-1$},\\
1+x & \text{if\ $-1\leq x\leq0$},\\
1 &\text{if\ $0<x$};
\end{array} %
\right.
\end{align*}
and the respective probability density functions (pdfs)
are the following:
\begin{align*}
w_{1,\alpha}(x) &= \frac{\alpha}{x}(\log x)^{-(\alpha+1)},\quad x\geq e;
\\
w_{2,\alpha}(x) &= \frac{\alpha}{x}(-\log x)^{(\alpha-1)},\quad
e^{-1} \leq x<1;
\\
w_3(x) &= \frac{1}{x^2}, \quad x>1;
\\
w_{4,\alpha}(x) &= \frac{-\alpha}{x}\bigl(-\log(-x)\bigr)^{-(\alpha+1)},
\quad -e^{-1} \leq x<0;
\\
w_{5,\alpha}(x) &= \frac{-\alpha}{x}\bigl(\log(-x)\bigr)^{(\alpha-1)},
\quad -e\leq x<-1;
\\
w_6(x) &= 1, \quad-1\leq x\leq0;
\end{align*}
where the pdfs are equal to 0 for the remaining values of $x$.

See also \cite{cormann2009generalizing} and \cite{ravi2011mises} for
more details on generalized log-Pareto distributions. The von-Mises
type sufficient conditions for {\sl p}-max stable laws were obtained in
\cite{mohan1998convergence}.\medskip

\noindent{\bf Von Mises-type parameterization of generalized log-Pareto
distributions.}
The von Mises-type parameterization for generalized log-Pareto
distributions is given by
\begin{align*}
V_1(x)&=1- \{1+\gamma\log x \}^{-1/\gamma},\\
&\quad x>0,(1+\gamma \log
x)>0, \mbox{ whenever } \gamma\geq0, \mbox{ and}\\
V_2(x)&=1- \bigl\{1-\gamma\log(-x) \bigr\}^{-1/\gamma},\\
&\quad x<0,
\bigl(1-\gamma \log(-x)\bigr)>0, \mbox{ whenever } \gamma\leq0,
\end{align*}
where the case $\gamma=0$ is interpreted as the limit as $\gamma
\rightarrow0$. Let $v_1$ and $v_2$ denote the densities of $V_1$ and
$V_2$, respectively. The dfs of generalized log-Pareto distributions
can be regained from $V_1$ and $V_2$ by the following identities:
\begin{align*}
W_{1,1/\gamma} (x)&= \left\{ %
\begin{array}{@{}ll}
0 & \text{if\ $x<e$},\\
V_1 (e^{-1/\gamma} x^{1/\gamma} ) & \text{if\ $e\leq x$},
\gamma>0;\\
\end{array} %
\right.
\\
%\end{eqnarray*}
%\begin{eqnarray*}
W_{2,-1/\gamma}(x) &= \left\{ %
\begin{array}{@{}ll}
0 & \text{if\ $x<e^{-1}$},\\
V_1 (e^{-1/\gamma} x^{-1/\gamma} ) & \text{if\ $e^{-1}
\leq x<1$},\\
1 &\text{if\ $1<x$}, \gamma>0;
\end{array} %
\right.
\\
W_{4,1/\gamma}(x) &= \left\{ %
\begin{array}{@{}ll}
0 &\text{if\ $x<-e^{-1}$},\\
V_2 (-e^{1/\gamma} (-x)^{1/\gamma} ) & \text{if\ $-e^{-1}\leq x<0$},\\
1 & \text{if\ $0<x$}, \gamma<0;
\end{array} %
\right.
\\
W_{5,-1/\gamma}(x)&= \left\{ %
\begin{array}{@{}ll}
0 & \text{if\ $x<-e$},\\
V_2 (-e^{1/\gamma} (-x)^{1/\gamma} ) & \text{if\ $-e\leq
x<-1$},\\
1 & \text{if\ $-1<x$}, \gamma<0.
\end{array} %
\right.
\end{align*}
Note that $\lim_{\gamma\rightarrow0}V_1(x)=W_3(x),
x>1$, and $\lim_{\gamma\rightarrow0}V_2(x)=W_6(x), x\in
[-1,0]$.\medskip

\noindent{\bf Graphical representation of generalized log-Pareto pdfs.}
\begin{figure}[t!]
\includegraphics{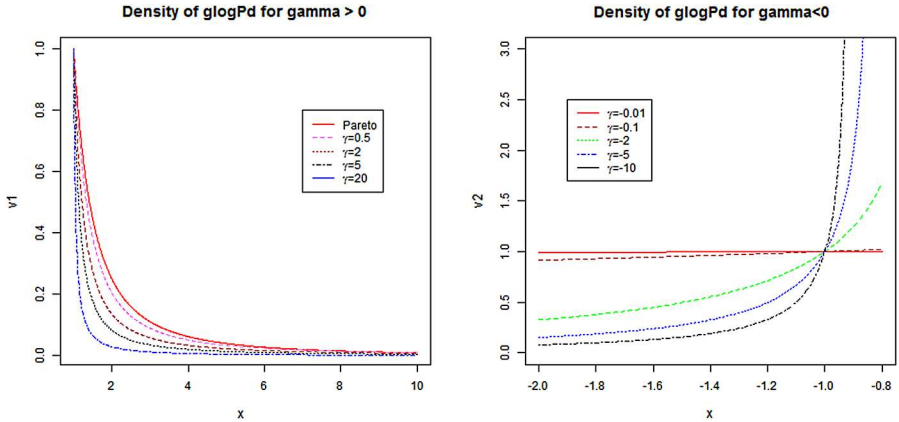}
\caption{Densities of glogPds}\label{1}
\end{figure}
In Fig. \ref{1}, observe that the pdfs $v_1$ approach the standard Pareto pdf as $\gamma
\downarrow0$, and the pdfs $v_2$ approach the standard uniform pdf as
$\gamma\uparrow0$.

The Hellinger distance, also called the Bhattacharya distance, is used
to quantify the similarity between two probability distributions, and
this was defined in terms of the Hellinger integral introduced by \cite
{hellinger1909neue}.
In view of statistical applications, the distance between the exact and
the limiting distributions is measured using the Hellinger distance.
Inference procedures based on the Hellinger distance provide
alternatives to likelihood-based methods. The minimum Hellinger
distance estimation with inlier modification was studied in~\cite
{patra2008minimum}. In \cite{reiss1989approximate}, the weak
convergence of distributions of extreme order statistics (defined later
in Section 2) was examined.

In the next section, we study the variational distance between the
exact and asymptotic distributions of power normalized partial maxima
of a random sample and the Hellinger distance between these. The
results obtained here are similar to those in \cite{reiss1989approximate}.

\section{Hellinger and variational distances for sample maxima}

\noindent We recall a few definitions for convenience.\medskip

\noindent{\bf Weak domain of attraction.} If a df $F$ satisfies (\ref
{Intro1}) for some norming constants and nondegenerate df $H$, then $F$
is said to belong to the weak domain of attraction of $H$.\medskip

\noindent{\bf Strong domain of attraction.} A df $F$ is said to belong
to the strong domain of attraction of a nondegenerate df $H$
if\vadjust{\eject}
\[
\lim_{n\rightarrow\infty}\sup_{B}\biggl\llvert P \biggl(
\biggl\llvert \frac
{X_{n:n}}{\alpha_n}\biggr\rrvert ^{1/\beta_n}\mbox{sign}(X_{n:n})
\in B \biggr)-H(B)\biggr\rrvert = 0, \;
\]
where sup is taken over all Borel sets $B$ on $R$.\medskip

\noindent{\bf Limit law for the $k$th largest order statistic} \cite
{mohan1993max}. Let $X_{1:n} \leq\cdots\leq X_{n:n}$ denote the
order statistics from a random sample $X_1, \ldots, X_n$, and for
$i = 1, \ldots, 6$, let
\[
\lim_{n\rightarrow\infty} P \biggl(\biggl\llvert \frac{X_{n:n}}{\alpha_n}\biggr
\rrvert ^{^{1/\beta_n}}\mbox{sign}(X_{n:n})\leq x \biggr)=
H_{i,\alpha}(x).
\]
Then it is well known that, for integer $k \geq1$,
\begin{align}
\label{kth}
\nonumber
\lim_{n\rightarrow\infty}P \biggl(\biggl\llvert
\frac{X_{n-k+1:n}}{\alpha_n}\biggr\rrvert ^{^{1/\beta_n}}\mbox {sign}(X_{n-k+1:n})\leq
x \biggr)&= H_{i,\alpha}(x)\sum_{j=0}^{k-1}
\frac
{  (-\log H_{i,\alpha}(x) )^j}{j!}
\\
&= H_{i,\alpha,k}(x),\quad \mbox{say.}
\end{align}

\noindent{\bf Hellinger distance} \cite{reiss1989approximate}.
Given dfs $F$ and $G$ with Lebesgue densities $f$ and~$g$, the
Hellinger distance between
$F$ and $G$, denoted $H^*(F,G)$, is defined as
\begin{equation}
\label{3.1} H^*(F,G)= \Biggl(\int_{-\infty}^{\infty}
\bigl(f^{1/2}(x)-g^{1/2}(x) \bigr)^2dx
\Biggr)^{1/2}.
\end{equation}
The results in this section will be proved for the $p$-max stable law
$H_{2,1}(\cdot)$,
and the other cases can be deduced by using the transformation
$T(x)=T_{i,\alpha}(x)$ given by
$T_{i,\alpha}(x)=H^{-1}_{i,\alpha}(x)\circ H_{2,1}(x)=H^{-1}_{i,\alpha
}(x), x\in(0,1)$ with
$T_{1,\alpha}(x)=\exp ((-\log x)^{-1/\alpha} )$,
$T_{2,\alpha}(x)=\exp (-(-\log x)^{1/\alpha} ),\;T_3(x)=-\frac
{1}{\log x},\;$
$T_{4,\alpha}(x)=-\exp (-\log x)^{-1/\alpha} ),$
$T_{5,\alpha}(x)=-\exp ((-\log x)^{1/\alpha} )$, and
$T_6(x)=\log x$.

We assume that the underlying pdf $f$ is of the form
$f(x)=w(x)e^{g(x)}$ where $g(x)\rightarrow0$ as $x\rightarrow r(H)
=\sup\{x:H(x)<1\}$, the right extremity of~$H$.
Equivalently, we may use the representation $f(x)=w(x)(1+g^*(x))$ by
writing $f(x)=w(x)e^{g(x)}=w(x)(1+(e^{g(x)}-1)),\;g(x)\rightarrow0$ as
$x\rightarrow r(F)$. The following result is on Hellinger distance,
and its proof is similar to that of Theorem 5.2.5 of \cite
{reiss1989approximate} and hence is omitted.

\begin{thm}\label{thm3.1}
Let $H$ be a p-max stable law as in \eqref{Intro1}, and $F$ be an
absolutely continuous df with pdf $f$ such that $f(x)>0$ for
$x_0<x<r(F)$ and $f(x)=0$ otherwise. Assume that $r(F)=r(H)$. Then
\begin{align}
H^*(F^n,H)&\leq \Biggl\{\int_{x_0}^{r(H)} \biggl(
\frac{nf(x)}{w(x)}-1-\log \biggl(\frac
{nf(x)}{w(x)} \biggr)dH(x)\notag\\
&\quad+2H(x_0)-H(x_0)\log H(x_0) \biggr) \Biggr
\}^{1/2}+\frac{c}{n},
\label{2.3}
\end{align}
where $c>0$ is a universal constant.
\end{thm}

\begin{thm}\label{2.2}
Suppose that $H$ is a p-max stable law as in \eqref{Intro1}, and $w(x),
T_{i,\alpha}(x)$ be the corresponding auxiliary functions with
$w(x)=h(x)/H(x)$ and $T_{i,\alpha}(x)=H^{-1}_{i,\alpha}(x)$, where $h$
denotes the pdf of $H$. Let the pdf $f$ of the df $F$ have the
representation\vadjust{\eject} $f(x)=w(x)e^{g_i(x)}, T(x_0)<x<r(H_i)$, for some $i$
and $=0 \mbox{ if } x>r(H_i)$, where $0<x_0<1$, and let $g_i$ satisfy
the condition
\begin{align}
\label{g} \big|g_i(x)\big|&\leq %
\begin{cases}
L(\log x)^{-\alpha\delta}& \emph{if}\ i=1,\\
L(-\log x)^{\alpha\delta} & \emph{if}\ i=2,\\
L\frac{1}{x^\delta} & \emph{if}\ i=3,\\
L(-\log(-x))^{-\alpha\delta}& \emph{if}\ i=4,\\
L(\log(-x))^{\alpha\delta} & \emph{if}\ i=5,\\
Lx^\delta& \emph{if}\ i=6,
\end{cases} %
\end{align}
where $L,\delta$ are positive constants. If
$F_n(x)=F(A_n|x|^{B_n}\mbox{\rm sign}(x))$ with
\[
A_n = %
\begin{cases}
1&\emph{if}\ i=1,2,3,4,\\
n&\emph{if}\ i=5,\\
1/n&\emph{if}\ i=6,
\end{cases} %
\quad \emph{and} \quad B_n = %
\begin{cases}
n^{-1/\alpha}&\emph{if}\ i=2,4,\\
n^{1/\alpha}&\emph{if}\ i=1,3,\\
1&\emph{if}\ i=5,6,
\end{cases} %
\]
then
\begin{align*}
H^*\bigl(F^n,H\bigr)&\leq %
\begin{cases}
Dn^{-\delta}& \emph{if}\ 0<\delta\leq1,\\
Dn^{-1} & \emph{if}\ \delta>1,
\end{cases} %
\end{align*}
where $D$ is a constant depending only on $x_0,L$, and $\delta$.
\end{thm}
\begin{proof}
Without loss of generality, we may assume that $H=H_{2,1}$. The other
cases can be deduced by using the transformation $T(x)=T_{i,\alpha
}(x)$. We apply Theorem \ref{thm3.1} with $x_{0,n}=x_0^n$, $\frac
{1}{2}<x_0<1$. Note that the term $2H_{2,1}(x_0^n)-H_{2,1}(x_0^n)\log
H_{2,1}(x_0^n)=x_0^n-x_0^n\log x_0^n$ can be neglected. Putting
$f_n(x)=f(x^{1/n})/n$, since $g$ is bounded on $(x_0,1)$, we have from
(\ref{2.3})
\[
H^*\bigl(F^n,H\bigr)\leq \Biggl\{\int_{x_0^n}^{1}
\biggl(\frac
{nf_n(x)}{w_{2,1}(x)}-1-\log \biggl(\frac{nf_n(x)}{w_{2,1}(x)} \biggr)
\biggr)dH_{2,1}(x) \Biggr\}^{1/2}+\dfrac{c}{n}.
\]
Then
\begin{align*}
&\int_{x_0^n}^{1} \left(\frac{nf_n(x)}{w_{2,1}(x)}-1-\log
 \left(\frac{nf_n(x)}{w_{2,1}(x)} \right) \right)dH_{2,1}(x)\\
&\quad= \int_{x_0^n}^{1} \biggl(\frac{f(x^{1/n})}{w_{2,1}(x)}-1-\log
\biggl(\frac
{nf(x^{1/n})}{w_{2,1}(x)} \biggr) \biggr)dH_{2,1}(x)
\\
&\quad=  \int_{x_0^n}^{1} \bigl(e^{h(x^{1/n})}-1-h
\bigl(x^{1/n}\bigr) \bigr)dx
\\
&\quad\leq \biggl(\frac{1}{2}+\frac{L(-\log x_0^{1/n})}{3!}+\cdots \biggr)\int
_{x_0^n}^{1} \bigl(h\bigl(x^{1/n}\bigr)
\bigr)^2dx
\\
&\quad\leq D^* \int_0^{1} \bigl(h
\bigl(x^{1/n}\bigr) \bigr)^2dx
\\
&\quad\leq D^*\frac{L^2}{2}\int_0^{1} \bigl(-
\log x^{1/n} \bigr)^{2\delta
}dx\\
&\quad= D^*\frac{L^2}{2}n^{-2\delta}\int_0^{\infty}e^{-y}y^{2\delta}dy
\\
&\quad= D^* \frac{L^2}{2}n^{-2\delta}\varGamma(2\delta+1),
\end{align*}
and $H^*(F_n^n,H(x))\leq (D^*\frac{L^2}{2}n^{-2\delta
}\varGamma(2\delta+1) )^{1/2}+\frac{c}{n}= (\frac{D^*}{2}
)^{1/2}Ln^{-\delta}(\varGamma(2\delta+1))^{1/2}+cn^{-1}$.
Hence,
\begin{align*}
H^*\bigl(F^n,H\bigr)&\leq %
\begin{cases}
Dn^{-\delta}& \text{if}\ 0<\delta\leq1,\\
Dn^{-1} & \text{if}\ \delta>1,
\end{cases} %
\end{align*}
where $D= (\frac{D^*}{2} )^{1/2}L\sqrt{\varGamma(2\delta+1)}$ is
a constant depending only on $x_0,\; L$, and $\delta$, and $\varGamma$ is
the gamma function.
\end{proof}

Theorem \ref{3.4} below gives the variational distance between exact
and approximate distributions of power normalized partial maxima. To
prove the result, we use the next result, the proof of which is similar
to that of Theorem 5.5.4 of \cite{reiss1989approximate} and hence is omitted.

\begin{thm}\label{3.3}
Let $H_j, j = 1, \ldots, 6$, denote the p-max stable laws as in \eqref
{Intro1}, and $H = H_{j,\alpha,k}$ denote the limit laws of the power
normalized $k$th largest order statistic as in \eqref{kth}. Let $F$ be
an absolutely continuous df with pdf $f$ such that $f(x)>0$ for
$x_0<x<r(F)$. Let $r(F)=r(H)$ and $w(x)=h(x)/H(x)$ on the support of
$H$, where $h$ is the pdf of $H$. Then
\begin{align*}
&\sup_B\big|P \big(\big(X_{n:n},\dots,X_{n-k+1:n}\big)\in B\big)-H_k(B)\big| \\
&\quad\leq\bigg(\sum_{j=1}^{k}\int_{x_0}^{r(H)} \left(\frac{nf(x)}{w(x)}-1-\log
 \left(\frac{nf(x)}{w(x)} \right) \right)dH_j(x) +H_k(x_0)+kH_{k-1}(x_0)\\
 &\qquad+\sum_{j=1}^{k-1}\int_{x_j>x_0,x_k<x_0}\log \left(\frac{nf(x_j)}{w(x_j)} \right)dH_k(x) \bigg)^{1/2}+\frac{ck}{n}.
\end{align*}
\end{thm}

\begin{thm} \label{3.4}
Let $H_j, j = 1, \ldots, 6$, denote the p-max stable laws as in \eqref
{Intro1} and $w(x), T_{i,\alpha}$ be the corresponding auxiliary
functions with $w(x)=h(x)/H(x)$ and $T_{i,\alpha}(x)=H_{i,\alpha
}^{-1}(x)$. Let the pdf $f$ of the absolutely continuous df $F$ satisfy
the representation $f(x)=w(x)e^{g_i(x)},\;T(x_0)<x<r(H)$, for some $i$
and $=0$ if $x>r(H)$, where $1/2<x_0<1$, and $g_i$ satisfy the
condition given in~\eqref{g}. Then
\begin{align*}
&\sup_B\bigg\llvert P \bigg\{ \bigg(\bigg\llvert \frac
{X_{n-j+1:n}}{A_n}\bigg\rrvert ^{1/B_n}\operatorname{sign}(X_{n-j+1:n}) \bigg)_{j=1}^k\in
B \bigg\}-H_k(B)\bigg\rrvert\\
&\quad\leq D\big((k/n)^\delta k^{1/2}+k/n\big),
\end{align*}
where $D$ is a constant depending on $x_0, L$, and $\delta$, and $A_n$
and $B_n$ are defined in Theorem \ref{2.2}.
\end{thm}

\begin{proof}
We prove the result for the particular case $H=H_{2,1}$. Applying
Theorem \ref{3.3} with $x_{0,n}=x_0^n,\; \frac{1}{2}<x_0<1$, we get
\begin{align*}
&\sup_B\big|P \big(\big(| X_{n:n}|^{n}\mbox{sign}(X_{n:n}),\dots,| X_{n-k+1:n}|^{n}\mbox{sign}(X_{n-k+1:n})\big)\in B \big)-H_k(B)\big|\\
&\quad\leq \bigg(\sum_{j=1}^{k}\int_{x_0^n}^{r(H)} \left(\frac{nf_n(x)}{w_{2,1}(x)}-1-\log\frac{nf_n(x)}{w_{2,1}(x)} \right) dH_j(x)+H_k(x_0^n)+kH_{k-1}(x_0^n)\\
&\qquad+\sum_{j=1}^{k-1}\int_{x_j>x_0^n,x_k<x_0^n}\log\frac{nf_n(x_j)}{w_{2,1}(x_j)}dH_k(x)
\bigg)^{1/2}+\dfrac{ck}{n}.
\end{align*}
Note that $H_k(x)=O((k/x)^m)$ uniformly in $k$ and $0<x<1$ for every
positive integer $m$. Moreover, since $h$ is bounded on $(x_0,1)$, we
have
\begin{align*}
&\sum_{j=1}^k\int_{x_0^n}^{r(H)} \left(\frac{nf_n(x)}{w_{2,1}(x)}-1-\log
\left(\frac{nf_n(x)}{w_{2,1}(x)} \right) \right)dH_j(x)\\
&\quad= \sum_{j=1}^k\int_{x_0^n}^{1}
\biggl(\frac{f(x^{1/n})}{w_{2,1}(x)}-1-\log \biggl(\frac{f(x^{1/n})}{w_{2,1}(x)} \biggr)
\biggr)h_j(x)dx
\\
&\quad=  \sum_{j=1}^k\int
_{x_0^n}^{1} \bigl(e^{h(x^{1/n})}-1-\log
\bigl(e^{h(x^{1/n})} \bigr) \bigr)h_j(x)dx
\\
&\quad=  \sum_{j=1}^k\int
_{x_0^n}^{1} \bigl(1+h\bigl(x^{1/n}\bigr)+
\dots -1-h\bigl(x^{1/n}\bigr) \bigr)h_j(x)dx
\\
&\quad\leq \biggl(\frac{1}{2}+\frac{L(-\log x_0^{1/n})}{3!}+\cdots \biggr)\sum
_{j=1}^k\int_{x_0^n}^{1}
\xch{\bigl(h\bigl(x^{1/n}\bigr)^2 \bigr)}{(\bigl(h\bigl(x^{1/n}\bigr)^2 \bigr)}\frac{(-\log
x)^{j-1}}{(j-1)!}dx
\\
&\quad\leq \biggl(\frac{1}{2}+\frac{L(-\log x_0^{1/n})}{3!}+\cdots \biggr)\sum
_{j=1}^k\int_0^{1}
\bigl(h\bigl(x^{1/n}\bigr)^2 \bigr)\frac{(-\log x)^{j-1}}{(j-1)!}dx\\
&\quad\leq \biggl(\frac{1}{2}+\frac{L(-\log x_0^{1/n})}{3!}+\cdots \biggr)\sum
_{j=1}^k\int_0^{1}
\frac{L^2(-\log x^{1/n})^{2\delta}}{2}\frac{(-\log
x)^{j-1}}{(j-1)!} dx
\\
&\quad= \biggl(\frac{1}{2}+\frac{L(-\log x_0^{1/n})}{3!}+\cdots \biggr)\frac
{L^2}{2}
\sum_{j=1}^k\frac{n^{-2\delta}}{\varGamma(j)}\int
_0^1(-\log x)^{2\delta+j-1}dx
\\
&\quad= D^*\sum_{j=1}^k\frac{n^{-2\delta}}{\varGamma(j)}\int
_0^\infty e^{-y} y^{2\delta+j-1}dx
\\
& \quad=  D^*n^{-2\delta}\sum_{j=1}^k
\frac{\varGamma(2\delta+j)}{\varGamma(j)},
\end{align*}
where $\varGamma$ is the gamma function. Now, note that (see, e.g., \cite
{erdaelyi1953higher}, p.~47)\vadjust{\eject}
\[
\sum_{j=1}^k\frac{\varGamma(2\delta+j)}{\varGamma(j)}\leq
D'\sum_{j=1}^kj^{2\delta}.
\]
Therefore,
\[
D^*n^{-2\delta}\sum_{j=1}^k
\frac{\varGamma(2\delta+j)}{\varGamma(j)}\leq D^*D'n^{-2\delta} \sum
_{j=1}^k j^{2\delta}\leq D^{**}n^{-2\delta}k^{2\delta+1},
\]
where $D^{**}=D^*D'$.
Hence,
\begin{align*}
&\sup_{B}\big|P \bigl(\bigl(X_{n:n}^n,
\dots,X_{n-k+1:n}^n\bigr)\in B \bigr)-H_k(B)\big|\\
&\quad\le \bigl(D^{**}n^{-2\delta}k^{2\delta+1} \bigr)^{1/2}+ck/n=D
\bigl((k/n)^\delta k^{1/2}+k/n\bigr),
\end{align*}
proving the result.
\end{proof}


\begin{thebibliography}{10}

%%% bbsrt2.pl, ver. 2.5.5, 2015.06.11
%b1 ###bbsrt2
%b2 ###
\bibitem{cormann2009generalizing}
\begin{barticle}
\bauthor{\bsnm{Cormann}, \binits{U.}},
\bauthor{\bsnm{Reiss}, \binits{R.-D.}}:
\batitle{Generalizing the Pareto to the log-Pareto model and statistical
inference}.
\bjtitle{Extremes}
\bvolume{12}(\bissue{1}),
\bfpage{93}--\blpage{105}
(\byear{2009}).
\bid{doi={10.1007/s10687-008-0070-\\6}, mr={2480725}}
\end{barticle}
\OrigBibText
\begin{barticle}
\bauthor{\bsnm{Cormann}, \binits{U.}},
\bauthor{\bsnm{Reiss}, \binits{R.-D.}}:
\batitle{Generalizing the Pareto to the log-Pareto model and statistical
inference}.
\bjtitle{Extremes}
\bvolume{12}(\bissue{1}),
\bfpage{93}--\blpage{105}
(\byear{2009})
\end{barticle}
\endOrigBibText
\bptok{structpyb}\endbibitem

%b2 ###bbsrt2
%b3 ###
\bibitem{erdaelyi1953higher}
\begin{bbook}
\bauthor{\bsnm{Erd{\'e}lyi}, \binits{A.}},
\bauthor{\bsnm{Magnus}, \binits{W.}},
\bauthor{\bsnm{Oberhettinger}, \binits{F.}},
\bauthor{\bsnm{Tricomi}, \binits{F.G.}}:
\bbtitle{Higher Transcendental Functions, vol. I.
Bateman Manuscript Project},
\bpublisher{McGraw--Hill},
\blocation{New York}
(\byear{1953}).
\end{bbook}
\OrigBibText
\begin{botherref}
\oauthor{\bsnm{Erd{\'e}lyi}, \binits{A.}},
\oauthor{\bsnm{Magnus}, \binits{W.}},
\oauthor{\bsnm{Oberhettinger}, \binits{F.}},
\oauthor{\bsnm{Tricomi}, \binits{F.G.}}:
Higher Transcendental Functions, vol. I.
Bateman Manuscript Project, McGraw-Hill, New York
(1953)
\end{botherref}
\endOrigBibText
\bptok{structpyb}\endbibitem

%b3 ###bbsrt2
%b4 ###
\bibitem{falk2010laws}
\begin{bbook}
\bauthor{\bsnm{Falk}, \binits{M.}},
\bauthor{\bsnm{H{\"u}sler}, \binits{J.}},
\bauthor{\bsnm{Reiss}, \binits{R.-D.}}:
\bbtitle{Laws of Small Numbers: Extremes and Rare Events}.
\bpublisher{Springer}
(\byear{2010}).
\bid{doi={10.1007/978-3-0348-0009-9}, mr={2732365}}
\end{bbook}
\OrigBibText
\begin{bbook}
\bauthor{\bsnm{Falk}, \binits{M.}},
\bauthor{\bsnm{H{\"u}sler}, \binits{J.}},
\bauthor{\bsnm{Reiss}, \binits{R.-D.}}:
\bbtitle{Laws of Small Numbers: Extremes and Rare Events}.
\bpublisher{Springer}
(\byear{2010})
\end{bbook}
\endOrigBibText
\bptok{structpyb}\endbibitem

%b4 ###bbsrt2
%b5 ###
\bibitem{hellinger1909neue}
\begin{barticle}
\bauthor{\bsnm{Hellinger}, \binits{E.}}:
\batitle{Neue Begr{\"u}ndung der Theorie quadratischer Formen von
unendlichvielen Ver{\"a}nderlichen.}
\bjtitle{J. Reine Angew. Math.}
\bvolume{136},
\bfpage{210}--\blpage{271}
(\byear{1909}).
\bid{doi={10.1515/crll.1909.136.210}, mr={1580780}}
\end{barticle}
\OrigBibText
\begin{barticle}
\bauthor{\bsnm{Hellinger}, \binits{E.}}:
\batitle{Neue Begr{\"u}ndung der Theorie quadratischer Formen von
unendlichvielen Ver{\"a}nderlichen.}
\bjtitle{Journal f{\"u}r die reine und angewandte Mathematik}
\bvolume{136},
\bfpage{210}--\blpage{271}
(\byear{1909})
\end{barticle}
\endOrigBibText
\bptok{structpyb}\endbibitem

%b5 ###bbsrt2
%b6 ###
\bibitem{mohan1993max}
\begin{barticle}
\bauthor{\bsnm{Mohan}, \binits{N.}},
\bauthor{\bsnm{Ravi}, \binits{S.}}:
\batitle{Max domains of attraction of univariate and multivariate
$p$-max stable laws}.
\bjtitle{Theory Probab. Appl.}
\bvolume{37}(\bissue{4}),
\bfpage{632}--\blpage{643}
(\byear{1993}).
\bid{doi={10.1137/1137119}, mr={1210055}}
\end{barticle}
\OrigBibText
\begin{barticle}
\bauthor{\bsnm{Mohan}, \binits{N.}},
\bauthor{\bsnm{Ravi}, \binits{S.}}:
\batitle{Max domains of attraction of univariate and multivariate
$p$-max stable
laws}.
\bjtitle{Theory of Probability and Its Applications}
\bvolume{37}(\bissue{4}),
\bfpage{632}--\blpage{643}
(\byear{1993})
\end{barticle}
\endOrigBibText
\bptok{structpyb}\endbibitem

%b6 ###bbsrt2
%b7 ###
\bibitem{mohan1998convergence}
\begin{barticle}
\bauthor{\bsnm{Mohan}, \binits{N.}},
\bauthor{\bsnm{Subramanya}, \binits{U.}}:
\batitle{On the convergence of the density of the power normalized maximum}.
\bjtitle{Calcutta Stat. Assoc. Bull.}
\bvolume{48}(\bissue{189-19}),
\bfpage{13}--\blpage{20}
(\byear{1998}).
\bid{mr={1677401}}
\end{barticle}
\OrigBibText
\begin{barticle}
\bauthor{\bsnm{Mohan}, \binits{N.}},
\bauthor{\bsnm{Subramanya}, \binits{U.}}:
\batitle{On the convergence of the density of the power normalized maximum}.
\bjtitle{Calcutta Statist. Assoc. Bull.}
\bvolume{48}(\bissue{189-19}),
\bfpage{13}--\blpage{20}
(\byear{1998})
\end{barticle}
\endOrigBibText
\bptok{structpyb}\endbibitem

%b7 ###bbsrt2
%b8 ###
\bibitem{patra2008minimum}
\begin{barticle}
\bauthor{\bsnm{Patra}, \binits{R.K.}},
\bauthor{\bsnm{Mandal}, \binits{A.}},
\bauthor{\bsnm{Basu}, \binits{A.}}:
\batitle{Minimum Hellinger distance estimation with inlier modification}.
\bjtitle{Sankhy{\=a}, Indian J. Stat. B (2008)},
\bvolume{70}(\bissue{2}),
\bfpage{310}--\blpage{322}
(\byear{2008}).
\bid{mr={2563992}}
\end{barticle}
\OrigBibText
\begin{botherref}
\oauthor{\bsnm{Patra}, \binits{R.K.}},
\oauthor{\bsnm{Mandal}, \binits{A.}},
\oauthor{\bsnm{Basu}, \binits{A.}}:
Minimum Hellinger distance estimation with inlier modification.
Sankhy{\=a}: The Indian Journal of Statistics, Series B (2008-),
\textbf{70}(2),
310--322
(2008)
\end{botherref}
\endOrigBibText
\bptok{structpyb}\endbibitem

%b8 ###bbsrt2
%b1 ###
\bibitem{Praveena}
\begin{botherref}
\oauthor{\bsnm{Praveena}, \binits{A.S.}}:
A Study of subexponential distributions and some problems in extreme value
theory.
Unpublished Ph.D.\ thesis, University of Mysore (2010).
\end{botherref}
\OrigBibText
\begin{botherref}
\oauthor{\bsnm{Praveena}, \binits{A.S.}}:
A Study of subexponential distributions and some problems in extreme value
theory.
Unpublished Ph.D.\ thesis, University of Mysore (2010).
\end{botherref}
\endOrigBibText
\bptok{structpyb}\endbibitem

%b9 ###bbsrt2
%b9 ###
\bibitem{ravi2011mises}
\begin{barticle}
\bauthor{\bsnm{Ravi}, \binits{S.}},
\bauthor{\bsnm{Praveena}, \binits{A.S.}}:
\batitle{On von Mises type conditions for $p$-max stable laws, rates of
convergence and generalized log Pareto distributions}.
\bjtitle{J. Stat. Plan. Inference}
\bvolume{141}(\bissue{9}),
\bfpage{3021}--\blpage{3034}
(\byear{2011}).
\bid{doi={10.1016/j.jspi.2011.03.024}, mr={2796008}}
\end{barticle}
\OrigBibText
\begin{barticle}
\bauthor{\bsnm{Ravi}, \binits{S.}},
\bauthor{\bsnm{Praveena}, \binits{A.S.}}:
\batitle{On von Mises type conditions for $p$-max stable laws, rates of
convergence and generalized log Pareto distributions}.
\bjtitle{Journal of Statistical Planning and Inference}
\bvolume{141}(\bissue{9}),
\bfpage{3021}--\blpage{3034}
(\byear{2011})
\end{barticle}
\endOrigBibText
\bptok{structpyb}\endbibitem

%b10 ###bbsrt2
%b10 ###
\bibitem{reiss1989approximate}
\begin{bbook}
\bauthor{\bsnm{Reiss}, \binits{R.}}:
\bbtitle{Approximate Distributions of Order Statistics. With
Applications to
Nonparametric Statistics}.
\bsertitle{Springer Series in Statistics}.
\bpublisher{Springer}
(\byear{1989}).
\bid{doi={10.1007/978-1-4613-9620-8}, mr={0988164}}
\end{bbook}
\OrigBibText
\begin{bbook}
\bauthor{\bsnm{Reiss}, \binits{R.}}:
\bbtitle{Approximate Distributions of Order Statistics. With
Applications to
Nonparametric Statistics. Springer Series in Statistics}.
\bpublisher{Springer}
(\byear{1989})
\end{bbook}
\endOrigBibText
\bptok{structpyb}\endbibitem



\end{thebibliography}
\end{document}